\documentclass[12pt,a4paper]{amsart}
\usepackage{graphicx,epsfig}
\usepackage{amssymb}
\usepackage{todonotes}
\usepackage[polish,english]{babel}
\usepackage{polski}

\usepackage{bbm} 

\def\qed{\hfill{ $\Box$ }}
\def\cd{\int_{-\infty}^{a}}
\def\cg{\int_{a}^{\infty}}
\def\gp{\mathbf{P}}
\def\sk{\vartriangle_{\alpha}}
\def\tp{\tau_{a}^{+}}

\def\tpw{\widetilde{\tau}_{a}^{+}}
\def\tmw{\widetilde{\tau}_{a}^{-}}

\newenvironment{Example}{\begin{Ex}\rm}{\smallskip\end{Ex}}

\newenvironment{namelist}[1]{%
\begin{list}{}
     {
      
      \settowidth{\labelwidth}{#1}
      \setlength{\leftmargin}{1.1\labelwidth}
               }
      }{%
\end{list}}

\newtheorem{defn}{Definition}
\newtheorem{thm}{Theorem}
\newtheorem{lem}{Lemma}
\newtheorem{prop}{Proposition}
\newtheorem{cor}{Corollary}
\newtheorem{Ex}{Example}[section]

\begin{document}

\title[Fluctuations of Kendall random walks]%
      {Fluctuations of extremal Markov chains driven by the Kendall convolution}
\author[Jasiulis-Go{\l}dyn, Omey, Staniak]{Barbara Jasiulis-Go{\l}dyn$^1$,  Edward Omey$^2$, Mateusz Staniak$^3$}
\thanks{$^{1,3}$ Institute of Mathematics, University of Wroc{\l}aw, pl. Grunwaldzki 2, 50-384 Wroc{\l}aw, Poland, \\
E-mail: $^1$Barbara.Jasiulis@math.uni.wroc.pl, 
$^3$Mateusz.Staniak@math.uni.wroc.pl\\
$^{2}$Faculty of Economics and Business-Campus Brussels, KU Leuven, Warmoesberg 26, 1000 Brussels, Belgium, e-mail: edward.omey@kuleuven.be 
           }

\maketitle

\begin{abstract}

The paper deals with fluctuations of Kendall random walks, which are extremal Markov chains and iterated integral transforms with the Williamson kernel $\Psi(t) = \left(1-|t|^{\alpha}\right)_+$, $\alpha>0$. We obtain the joint distribution of the first ascending ladder epoch and height over any level $a \geq 0$ and distribution of maximum and minimum for these extremal Markovian sequences solving recursive integral equations.
We show that distribution of the first crossing time of level $a \geq0$ is a mixture of geometric and negative binomial distributions. The Williamson transform is the main tool for considered problems connected with the Kendall convolution. All results are described by the Williamson transform of the unit step distribution of Kendall random walks. Using regular variation, we investigate the asymptotic properties of the maximum distribution.

\end{abstract}

\section{Introduction}

Addition of independent random variables and corresponding operation on their measures - convolution - is one of the most commonly occurring operations in probability theory and applications. Classical convolution is a special case of much more general operation called a generalized convolution.

The origin of generalized convolutions can be found in delphic semigroups (\cite{Gilewski1, Kendall}).
Inspired by Kingman's study of spherical random walks (\cite{King}), Urbanik introduced the notion of a generalized convolution for measures concentrated on a positive half-line in a series of papers \cite{Urbanik64}.
This definition was extended to symmetrical measures on $\mathbb{R}$ by Jasiulis-Go{\l}dyn in \cite{Jasiulis2010}.
Generalized convolutions were explored with the use of regular variation (\cite{Bin71, Bin84, renewalKendall}) and were used to construct L\'evy processes and stochastic integrals (\cite{BJMR}, \cite{Ruben1}). 
In the theory of generalized convolutions, we create new mathematical objects that have potential in applications. It is enough to look at the case of the maximum convolution corresponding to extreme value theory (\cite{Embrechts}). Currently, limit distribution for extremes, i.e. generalized extreme value distribution (Frech\'et, Gumbel, Weibull) is commonly used for modeling rainfall, floods, drought, cyclones, extreme air pollutants, etc. Random walks with respect to generalized convolutions form a class of extremal Markov chains (see \cite{Alpuim, BJMR, renewalKendall}).
Studying them in the appropriate algebras will be a meaningful contribution to extreme value theory (\cite{Embrechts}). Kendall random walks (\cite{KendallWalk, renewalKendall}), which are the main objects of investigations in this paper, are related to maximal processes (and thus maximum convolution), Pareto processes (\cite{Ferreira}) or pRARMAX processes (\cite{pRARMAX}).
One of the differences lies in the fact that in this case values of the process are randomly multiplied by a heavy-tailed Pareto random variables, which results in even more extreme behavior than in the case of the classical maximum process.

Generalized convolutions have connections with the theory of weakly stable measures (\cite{Jasiulis2010,misjas2}) and non-commutative probability (\cite{JasKula}). By the Williamson transform one can also find some connections with copula theory (\cite{Neslehova1, Neslehova2}). In \cite{MATKAP} connections  between the Kendall convolution and $\alpha$-slash distribution were described, as a consequence the potential of generalized convolutions in applications was shown. Moreover, at the same paper, the authors mentioned the Sibuya distribution in the context of the Kendall and Kucharczak-Urbanik convolutions.

\indent 

Fluctuations of classical random walks and L\'evy processes were widely described in literature (see e.g. \cite{Feller, Sato}) and they are still the object of interest for scientists \cite{Kennedy, Kyprianou, KyprPalm, Lachal, Nakajima}. This paper is a continuation of research initiated in \cite{factor}. The main result of this paper is a description of fluctuations of random walks generated by a particular generalized convolution - the Kendall convolution - in terms of the first ladder moment (epoch) and the first ladder height of the random walk over any level $a \geq 0$.
It turns out that the distribution of the first ladder moment is a mixture of three negative binomial distributions where coefficients and parameters depend on the unit step distribution. We also present distribution of maxima and minima for the Kendall random walks in terms of the Williamson transform and unit step cumulative distribution function. Description of the behavior of the extremes of the stochastic processes is an analogue of the Pollaczek-Khintchine equation from classical theory. 

\vspace{3mm}
Organization of the paper:
We begin Section 2 by recalling the Kendall convolution (\cite{KendallWalk,factor}), which is an example of generalized convolution in the sense defined by K. Urbanik \cite{Urbanik64}. This convolution is quite specific because the result of two point-mass probability measures is a convex linear combination of measure concentrated at one point and the Pareto distribution. Next, we present the definition and main properties of the Williamson transform (\cite{factor,Williamson}), which is analog of characteristic function in the Kendall convolution algebra. This transform is very easy to invert and allows us to get many results for extremal Markov sequences of the Kendall type.

In the third section, we consider the first ascending ladder epoch over any level $a \geq 0$ and prove that its distribution is a convex linear combination of geometrical and negative binomial distributions. Section 4 consists of two parts: the distribution of the first ladder height and the maximum and minimum distributions for studied stochastic processes.

\vspace{3mm}

Notation: The distribution of the random
element $X$ is denoted by $\mathcal{L}(X)$. For a probability measure $\lambda$ and $a \in
\mathbb{R}_+$ the rescaling operator is given by $\mathbf{T}_a \lambda = \mathcal{L}(aX)$ if $\lambda = \mathcal{L}(X)$. By $\mathcal{P}$ we denote family of all probability measures on the Borel subsets of $\mathbb{R}$, while by $\mathcal{P}_s$  we denote symmetric probability measures on $\mathbb{R}$. For abbreviation the set of all natural numbers including zero is denoted by $\mathbb{N}_0$. Additionally $\widetilde{\pi}_{2\alpha}$ denotes a Pareto random measure with the density function $\widetilde{\pi}_{2\alpha}\, (dy) = \alpha |y|^{-2\alpha-1} \pmb{1}_{[1,\infty)}(|y|)\, dy$. In general, by $\widetilde{\nu}$ we denote symmetrization of a  probability measure $\nu$. In this paper, we usually consider symmetric probability measures assuming that $\nu \in \mathcal{P}_s$.
Finally, for the tail distribution of $\nu$ with cumulative distribution function $F$ we use the standard notation $\overline{F}$.

\vspace{3mm}

We study positive and negative excursions for the Kendall random walk $\{X_n \colon n \in \mathbb{N}_0\}$ which is defined by the following construction:
\begin{defn} Stochastic process $\{X_n \colon n \in \mathbb{N}_0\}$ is a discrete time Kendall random walk  with parameter $\alpha>0$ and step distribution $\nu$ if there exist
\begin{namelist}{ll}
\item[\bf 1.] $(Y_k)$ i.i.d. random variables with distribution $\nu \in \mathcal{P}_s$,
\item[\bf 2.] $(\xi_k)$ i.i.d. random variables with uniform distribution on $[0,1]$,
\item[\bf 3.] $(\theta_k)$ i.i.d. random variables with the symmetric Pareto distribution with the density $\widetilde{\pi}_{2\alpha}\, (dy) = \alpha |y|^{-2\alpha-1} \pmb{1}_{[1,\infty)}(|y|)\, dy$,
\item[\bf 4.]  sequences $(Y_k)$, $(\xi_k)$ and $(\theta_k)$ are independent,
\end{namelist}
such that
$$
X_0 = 1, \quad X_1 = Y_1, \quad  X_{n+1} = M_{n+1} r_{n+1} \left[ \mathbf{I}(\xi_n > \varrho_{n+1}) + \theta_{n+1} \mathbf{I}(\xi_n < \varrho_{n+1})\right],
$$
where $\theta_{n+1}$ and $M_{n+1}$ are independent,
$$
M_{n+1} = \max\{ |X_n|,|Y_{n+1}|\}, \quad m_{n+1} = \min\{ |X_n|,|Y_{n+1}|\}, \quad \varrho_{n+1} = \frac{m_{n+1}^{\alpha}}{M_{n+1}^{\alpha}}
$$ 
and
$$
r_{n+1} = \left\{sgn(u) : \max \left\{|X_n|,|Y_{n+1}| \right\} = |u| \right\}.
$$

\end{defn}
The Kendall random walk is extremal Markov chain with $X_0\equiv 0$ and the transition probabilities
$$
P_n(x, A) = \mathbf{P}\left\{ X_{n+k} \in A
\big| X_k = x \right\} = \delta_x \vartriangle_{\alpha} \nu^{\vartriangle_{\alpha} n}, \quad n,k \in \mathbb{N},
$$
where measure $\nu \in \mathcal{P}_s$ is called the step distribution.
Construction and some basic properties of this particular  process are described in  \cite{BJMR, KendallWalk, renewalKendall, misjas3}.

\section{Williamson transform and Kendall convolution }

\subsection{ The Kendall convolution approach}
The stochastic process considered here was constructed using the Kendall convolution defined in the following way:
\begin{defn}
Commutative and associative binary operation  $\vartriangle_{\alpha} \colon \mathcal{P}_s^2 \rightarrow \mathcal{P}_s$ defined for discrete measures by
$$
\widetilde{\delta}_x \vartriangle_{\alpha} \widetilde{\delta}_y := T_M \left( \varrho^{\alpha} \widetilde{\pi}_{2\alpha} + (1- \varrho^{\alpha}) \widetilde{\delta}_1 \right)
$$
where $M = \max\{|x|, |y|\}$, $m = \min\{|x|, |y|\}$, $\varrho = {m/M}$, we call the Kendall convolution. The extension of $\vartriangle_{\alpha}$ to the whole $\mathcal{P}_s$ is given by
$$
\nu_1 \vartriangle_{\alpha} \nu_2 (A)  = \int_{\mathbb{R}^2} \widetilde{\delta}_x \vartriangle_{\alpha} \widetilde{\delta}_y (A)\, \nu_1(dx) \nu_2(dy).
$$
\end{defn}
Notice that the operation $\vartriangle_{\alpha}$ is a generalized convolution in the sense introduced by Urbanik (see \cite{Jasiulis2010,Urbanik64,vol2}) having the following properties:
\begin{itemize}
\item $\nu \vartriangle_{\alpha} \delta_0 = \nu$ for each $\nu \in \mathcal{P}_s$;
\item $\left( p \nu_1 + (1-p)\nu_2 \right) \vartriangle_{\alpha} \nu = p
\left( \nu_1 \vartriangle_{\alpha} \nu \right) + (1-p) \left( \nu_2 \vartriangle_{\alpha}
\nu \right)$ for each $p \in [0,1]$ and each $\nu, \nu_1, \nu_2 \in \mathcal{P}_s$.
\item if $\lambda_n \rightarrow \lambda$ and $\nu_n \rightarrow \nu$, then $(\lambda_n \vartriangle_{\alpha} \nu_n) \rightarrow (\lambda \vartriangle_{\alpha} \nu)$, where $\rightarrow$ denotes the weak convergence;
\item $T_a \bigl( \nu_1 \vartriangle_{\alpha}\nu_2 \bigr)
= \bigl(T_a \nu_1\bigr) \vartriangle_{\alpha} \bigl(
T_a \nu_2\bigr)$ for each $\nu_1, \nu_2 \in \mathcal{P}_s$.
\end{itemize}
The Kendall convolution is strictly connected with the Williamson transform, which plays similar role to characteristic function in the classical algebra.

\begin{defn}
By the Williamson transform we understand the operation $\nu \rightarrow \widehat{\nu}$ given by
$$
\widehat{\nu}(t) = \int_{\mathbb{R}} \left( 1 - |xt|^{\alpha} \right)_+ \nu(dx), \quad \nu \in \mathcal{P}_s,
$$
where $a_+ = a$ if $a\geqslant 0$ and $a_+=0$ otherwise.
\end{defn}

For convenience we use the following notation:
$$
\Psi(t) = \left( 1 - |t|^{\alpha} \right)_+, \quad G(t) = \widehat{\nu}(1/t).
$$
The next lemma is almost evident and well known. It provides the inverse of the Williamson transform, which is surprisingly simple.
\begin{lem} \label{lem:1}
The correspondence between measure $\nu \in \mathcal{P}_s$ and its Williamson transform is $1-1$. Moreover, denoting by $F$ the cumulative distribution function of $\nu$, $\nu(\{0\}) = 0$, we have
$$
F(t) = \left\{ \begin{array}{lcl}
 \frac{1}{2\alpha} \left[ \alpha (G(t) + 1) + t G'(t) \right] & \hbox{if} & t>0; \\
 1 - F(-t) & \hbox{if} & t<0.
 \end{array} \right.
$$
except for the countable many $t \in \mathbb{R}$.
\end{lem}
For details of the proof of the above Lemma see \cite{factor}.

As we mentioned above the Williamson transform (\cite{Williamson}) plays the same role for the Kendall convolution as the Fourier transform for the classical convolution (for proof see Proposition 2.2. in \cite{factor}), i.e.
\begin{prop}\label{prop:1}
Let $\nu_1, \nu_2 \in \mathcal{P}_s$ be probability measures with Williamson transforms $\widehat{\nu_1}, \widehat{\nu_2}$. Then
$$
\int_{\mathbb{R}} \Psi(xt) \bigl(\nu_1 \vartriangle_{\alpha} \nu_2 \bigr)(dx) = \widehat{\nu_1}(t) \widehat{\nu_2}(t).
$$
\end{prop}

The following fact is a simple consequence of Lemma \ref{lem:1} and Proposition \ref{prop:1}.
\begin{prop}\label{prop:2} 
Let $\nu \in \mathcal{P}$. For each natural number $n\geqslant 2$ the cumulative distribution function $F_n$ of measure $\nu^{ \vartriangle_{\alpha} n}$ is equal
$$
F_n(t) = \frac{1}{2} \left[  G(t)^n + 1 + n  G(t)^{n-1} H(t) \right], \quad t>0,
$$
where
$$
H(t) =2F(t)-G(t)-1 = t^{-\alpha} \int\limits_{-t}^{t} |x|^{\alpha}\nu(dx)
$$
and $F_n(t) = 1 - F_n(-t)$ for $t<0$, where $G(t) = \widehat{\nu}(1/t)$.
\end{prop}

\noindent
{\bf Proof.}
At the beginning, it is worth noting that 
$$
G_n(t) = G(t)^n.
$$
Then by Lemma \ref{lem:1} we arrive at the following formula:
$$
F_n(t) = \frac{1}{2\alpha} \left[ \alpha \bigl(G(t)^n + 1\bigr) + t n  G(t)^{n-1} G'(t) \right]
$$
for $t>0$ and we also have
$$
G'(t) = \frac{\alpha}{t} H(t),
$$
which ends the proof.

\qed

\begin{Example}\label{ex:delta1} 
Let $\nu = \widetilde{\delta}_1$. Then 
$$
G(t) = \left(1-|t|^{-\alpha}\right)_+
$$
and 
$$
dF_n (t) =\frac{\alpha n(n-1)}{2|t|^{2\alpha+1}} \left(1-|t|^{-\alpha}\right)^{(n-1)} \pmb{1}_{[1,\infty)}(|t|)\, dt.
$$
\end{Example} 

\begin{Example}\label{ex:stable} 
For Kendall random walk with unit step distribution $X_1 \sim \nu_{\alpha} \in\mathcal{P}$ such that $E|X_1|^{\alpha} = m_{\alpha} < \infty$, stable distribution has the following density
$$
\nu_{\alpha}(dx) = \frac{\alpha}{2} m_\alpha |x|^{-2(\alpha +1)} \exp\{-m_{\alpha} |x|^{-\alpha}\} dx.
$$
Then
$$
F_1(t) =  \left\{ \begin{array}{lcl}
 \frac{1}{2} + \frac{1}{2}\left( 1 + m_\alpha t^{-\alpha} \right) \exp\{-m_\alpha t^{-\alpha}\}& \hbox{if} & t>0; \\
 1 - F(-t) & \hbox{if} & t<0
 \end{array} \right.
$$
and
\begin{eqnarray*}
G(t) &=& \exp\left\{ - m_\alpha |t|^{-\alpha}\right\} \\
F_n(t) &=&  \left\{ \begin{array}{lcl}
 \frac{1}{2} + \frac{1}{2}\left( 1 + n m_\alpha t^{-\alpha} \right) \exp\{-n m_\alpha t^{-\alpha}\}& \hbox{if} & t>0; \\
 1 - F(-t) & \hbox{if} & t<0.
 \end{array} \right.
\end{eqnarray*}
It is evident that we have:
$$
F_n(t) = F_1(n^{-1/\alpha} t).
$$
\end{Example}

\begin{Example}\label{ex:Pareto}
Let $\nu = \widetilde{\pi}_{2\alpha}$ for $\alpha \in (0,1]$. Since $\widetilde{\delta}_1 \vartriangle_{\alpha} \widetilde{\delta}_1 = \widetilde{\pi}_{2\alpha}$, then using Example 2.1 we arrive at:
$$
dF_n (t) =\frac{\alpha n(2n-1)}{|t|^{2\alpha+1}} \left(1-|t|^{-\alpha}\right)^{2(n-1)} \pmb{1}_{[1,\infty)}(|t|)\, dt.
$$
\end{Example}

The explicit formula for transition probabilities for Kendall random walks is  given by: 
\begin{lem}\label{lem:2}
For all $n\in\mathbb{N}$ and $x \geq 0$  
\begin{eqnarray*}
  \left(\delta_x \vartriangle_{\alpha} \nu^{\vartriangle_{\alpha} n}\right) \,(0,t) = P_n (x, [0,t)) =  \frac{1}{2}\left[\Psi\left(\frac{x}{t}\right)H_n(t) + G_n(t)\right]\mathbf{1}_{\{|x| < t \}},
\end{eqnarray*}
where
$$
H_n(t) = \Bigl(2F_n(t) - 1 - G_n(t)\Bigr).
$$
\end{lem}
\noindent
{\bf Proof.} By Lemma 3.1 in \cite{factor} we have
\begin{eqnarray*}
\left(\delta_x \vartriangle_{\alpha} \delta_y\right) (0,t) & = &  \frac{1}{2} \left( 1 - \left| \frac{xy}{t^2}\right|^{\alpha} \right) \mathbf{1}_{\{|x| < t, |y| < t\}} \\
 & & \hspace{-15mm} = \frac{1}{2}\left[ \Psi\left(\frac{x}{t}\right) + \Psi\left(\frac{y}{t}\right) - \Psi\left(\frac{x}{t}\right) \Psi\left(\frac{y}{t}\right)\right] \mathbf{1}_{\{|x| < t, |y| < t\}}, \\
 \left(\delta_x \vartriangle_{\alpha} \nu \right) \,(0,t) & = & P_1 (x, [0,t)) \\
 & & \hspace{-15mm} = \frac{1}{2}\left[\Psi\left(\frac{x}{t}\right)\Bigl(2F(t) - 1 - G(t)\Bigr) + G(t)\right]\mathbf{1}_{\{|x| < t \}}. 
\end{eqnarray*}
The transition probability can now be computed by replacing $\nu$ with $\nu^{\vartriangle_{\alpha} n}$ in the last formula.
\qed

In the following section, we will also need the formula for the integral 
$$
\int_{-\infty}^{a}\Psi\left(\frac{x}{t}\right)(\delta_{y}\vartriangle_{\alpha}\nu)(dx).
$$
In order to find it, we first need to find the following truncated moment of order $\alpha$.
\begin{lem}\label{lem:pomoc}
For all $n \in\mathbb{N}$ and $a >0$ 
\begin{eqnarray*}
\int_{0}^{a} x^{\alpha}(\delta_{y}\vartriangle_{\alpha}\nu)(dx) &=& \frac{1}{2} \left[ H(a) \left(a^{\alpha}-|y|^{\alpha}\right) +|y|^{\alpha} G(a) \right] \mathbbm{1}(|y| < a) \\
&=&  \frac{a^{\alpha}}{2} \left[ H(a) \Psi\left(\frac{y}{a}\right) + \left(1 - \Psi\left(\frac{y}{a}\right)\right) G(a) \right] \mathbf{1}_{\{|y| < a\}}.
\end{eqnarray*}
\end{lem}

\begin{proof}
By Lemma \ref{lem:2} we have
\begin{eqnarray*}
\left(\delta_y \vartriangle_{\alpha} \delta_z\right) (0,x) & = &  \frac{1}{2} \left( 1 - \left| \frac{yz}{x^2}\right|^{\alpha} \right) \mathbf{1}_{\{|y| < x, |z| < x\}}.
\end{eqnarray*}
Integrating by parts, we obtain
\begin{eqnarray*}
  & & \int_{0}^{a} x^{\alpha}(\delta_{y}\vartriangle_{\alpha} \delta_{z})(dx) = 
  a^{\alpha} \left(\delta_y \vartriangle_{\alpha} \delta_z\right) (0,a) - \int\limits_0^a \alpha x^{\alpha-1} \left(\delta_y \vartriangle_{\alpha} \delta_z\right) (0,x) (dx)\\
  & = & \left(\frac{1}{2} |y|^{\alpha} - \frac{|yz|^{\alpha}}{a^{\alpha}} + \frac{1}{2} |z|^{\alpha} \right) \mathbf{1}_{\{|y| < a, |z| < a\}}\\
\end{eqnarray*}
from which it follows that 
\begin{eqnarray*}
  & & \int_{0}^{a} x^{\alpha}(\delta_{y}\vartriangle_{\alpha}\nu)(dx) = 
  \int_{-\infty}^{\infty}\int_{0}^{a} x^{\alpha}(\delta_{y}\vartriangle_{\alpha}\delta_{z})(dx) \nu(dz)\\
  & = & \int_{-a}^{a} \left(\frac{1}{2} |y|^{\alpha} - \frac{|yz|^{\alpha}}{a^{\alpha}} + \frac{1}{2} |z|^{\alpha} \right) \nu(dz) \\
  & = & \frac{1}{2} \left[ H(a) \left(a^{\alpha}-|y|^{\alpha}\right) + |y|^{\alpha} G(a) \right] \mathbf{1}_{\{|y| < a\}} \\
  & = & \frac{a^{\alpha}}{2} \left[ H(a) \Psi\left(\frac{y}{a}\right) + \left(1 - \Psi\left(\frac{y}{a}\right)\right) G(a) \right] \mathbf{1}_{\{|y| < a\}}.
\end{eqnarray*}
\end{proof}

Now we can find the formula for $\int_{-\infty}^{a}\Psi\left(\frac{x}{t}\right)(\delta_{y}\vartriangle_{\alpha}\nu)(dx)$.
\begin{lem}\label{lem:psiat}
For all $t \geq a \geq 0$ the following equality holds.
\begin{eqnarray*}
  & & \int_{-\infty}^{a}\Psi\left(\frac{x}{t}\right)(\delta_{y}\vartriangle_{\alpha}\nu)(dx) = \\
  & = & \frac{1}{2}\Psi\left(\frac{y}{a}\right)M(a, t) + \frac{1}{2}G(a)\Psi\left(\frac{a}{t}\right)\mathbbm{1}(|y| < a) + \frac{1}{2}\Psi\left(\frac{y}{t}\right)G(t),
\end{eqnarray*}
where
$$
M(a, t) = H(a)\Psi\left(\frac{a}{t}\right) + \left(1 - \Psi\left(\frac{a}{t}\right)\right)G(a).
$$
\end{lem}
\begin{proof}
By Lemmas \ref{lem:2} and \ref{lem:pomoc} we have
\begin{eqnarray*}
  & & \int_{-\infty}^{a}\Psi\left(\frac{x}{t}\right)(\delta_{y}\vartriangle_{\alpha}\nu)(dx) \\
  & = & (\delta_{y}\vartriangle_{\alpha}\nu)(-t,a) - t^{-\alpha} \int_{-t}^{a}|x|^{\alpha}(\delta_{y}\vartriangle_{\alpha}\nu)(dx) \\
  & = & (\delta_{y}\vartriangle_{\alpha}\nu)(0,a) + (\delta_{y}\vartriangle_{\alpha}\nu)(0,t)  \\
  &-& t^{-\alpha} \int_{0}^{a} x^{\alpha}(\delta_{y}\vartriangle_{\alpha}\nu)(dx) - t^{-\alpha} \int_{0}^{t} x^{\alpha}(\delta_{y}\vartriangle_{\alpha}\nu)(dx) \\
  & = & \frac{1}{2}\left[\Psi\left(\frac{y}{a}\right)H(a) + G(a)\right]\mathbf{1}_{\{|y| < a \}} + \frac{1}{2}\left[\Psi\left(\frac{y}{t}\right)H(t) + G(t)\right]\mathbf{1}_{\{|y| < t \}} \\
  & - & \frac{a^{\alpha}}{2t^{\alpha}} \left[ H(a) \Psi\left(\frac{y}{a}\right) + \left(1 - \Psi\left(\frac{y}{a}\right)\right) G(a) \right] \mathbf{1}_{\{|y| < a\}} \\
  & - & \frac{1}{2} \left[ H(t) \Psi\left(\frac{y}{t}\right) + \left(1 - \Psi\left(\frac{y}{t}\right)\right) G(t) \right] \mathbf{1}_{\{|y| < t\}}
\end{eqnarray*}
from which we obtain the desired result by regrouping the terms and noticing that since $t \geq a$, $\left(\frac{a}{t}\right)^{\alpha} = 1 - \Psi\left(\frac{a}{t}\right)$.
\end{proof}

Let us notice that in particular for $t=a$
$$
\int_{-\infty}^{a}\Psi\left(\frac{x}{a}\right)(\delta_{y}\vartriangle_{\alpha}\nu)(dx) = G(a)\Psi\left(\frac{y}{a}\right),
$$
because it is the Williamson transform of the Kendall convolution of two measures: $\delta_{y}$ and $\nu$.

Based on previous results, we will find closed-form formulas for two more important integrals.
Let us define
\begin{eqnarray*}
I(n, a, t) &:=& \int_{-\infty}^{a}\ldots\int_{-\infty}^{a}\Psi\left(\frac{x_{n}}{t}\right)(\delta_{x_{n - 1}}\vartriangle_{\alpha}\nu)(dx_{n})\ldots\nu(dx_{1}), \\
II(n, a, t) &:=& \int_{-\infty}^{a}\ldots\int_{-\infty}^{a}\mathbbm{1}(|x_{n}| < t)(\delta_{x_{n - 1}}\vartriangle_{\alpha}\nu)(dx_{n})\ldots\nu(dx_{1})
\end{eqnarray*}
In both of these expressions we integrate $n$ times.

First, we will find $I(n, a, a)$ and $II(n, a, a)$, which is much simpler than the general case and will be used in following calculations.

\begin{lem} For all $n \geq 1$
$$
I(n, a, a) =  G(a)^{n}.
$$
\end{lem}
\begin{proof}
First, let us notice that
$$
I(1, a, a) = \int_{-\infty}^{a}\Psi\left(\frac{x_{n}}{a}\right)\nu(dx_{1}) = G(a) 
$$
by the definition of Williamson transform.
By Lemma \ref{prop:1} we have
\begin{eqnarray*}
I(n, a, a) &=& \int_{-\infty}^{a}\ldots\int_{-\infty}^{a}\Psi\left(\frac{x_{n}}{a}\right)(\delta_{x_{n - 1}}\vartriangle_{\alpha}\nu)(dx_{n})\ldots\nu(dx_{1}) \\
&=& \int_{-\infty}^{a}\ldots\int_{-\infty}^{a}G(a)\Psi\left(\frac{x_{n - 1}}{a}\right)(\delta_{x_{n - 2}}\vartriangle_{\alpha}\nu)(dx_{n - 1})\ldots\nu(dx_{1}) \\
&=& G(a)I(n - 1, a, a).
\end{eqnarray*}
It follows that $I(n, a, t)$ is geometric sequence with common ratio equal to $G(a)$.
\end{proof}

\begin{lem} For all $n \geq 1$
$$
II(n, a, a)  = G(a)^{n - 1}\left[nH(a) + G(a)\right].
$$
\end{lem}
\begin{proof}
First, let us notice that 
$$
II(1, a, a) = \int_{-\infty}^{a}\mathbbm{1}(|x_{n}| < a)\nu(dx_{1}) = 2F(a) - 1 = H(a) + G(a).
$$
By Lemma \ref{lem:2} sequence $II(n, a, a)$ solves the following recurrence equation.
\begin{eqnarray*}
II(n, a, a) &=& \int_{-\infty}^{a}\ldots\int_{-\infty}^{a}\mathbbm{1}(|x_{n}| < a)(\delta_{x_{n - 1}}\vartriangle_{\alpha}\nu)(dx_{n})\ldots\nu(dx_{1}) \\
&=& H(a)\int_{-\infty}^{a}\ldots\int_{-\infty}^{a}\Psi\left(\frac{x_{n - 1}}{a}\right)(\delta_{x_{n - 2}}\vartriangle_{\alpha}\nu)(dx_{n - 1})\ldots\nu(dx_{1}) \\
&+& G(a)\int_{-\infty}^{a}\ldots\int_{-\infty}^{a}\mathbbm{1}(|x_{n - 1}| < a)(\delta_{x_{n - 2}}\vartriangle_{\alpha}\nu)(dx_{n - 1})\ldots\nu(dx_{1}) \\
&=& H(a)I(n - 1, a, a) + G(a)II(n - 1, a, a) \\
&=& H(a)G(a)^{n - 1} + G(a)II(n - 1, a, a).
\end{eqnarray*}
On the other hand, we have
\small{\begin{eqnarray*}
& & H(a)G(a)^{n - 1} + G(a)II(n - 1, a, a) \\
&=& H(a)G(a)^{n - 1} + G(a)G(a)^{n - 2}\left[(n - 1)H(a) + G(a)\right] \\
&=& G(a)^{n - 1}\left[nH(a) + G(a)\right] = II(n, a, a),
\end{eqnarray*}}
which ends the proof.
\end{proof}

Using these results we can find the expression for $I(n, a, t)$.
\begin{thm}\label{thm:inat}
The integral $I(n, a, t)$ is given by 
$$
I(n, a, t) = C_{1}\left(\frac{G(t)}{2}\right)^{n - 1} + G(a)^{n}\left[C_{2}n + C_{3}\right],
$$
$G(t) \neq 2G(a)$, for $n \geq 2$ and by $\frac{G(t)}{2}  + \frac{1}{2} H(a) \Psi\left(\frac{a}{t}\right) + \frac{G(a)}{2}$ for $n = 1$,
where
\begin{eqnarray*}
\begin{cases}
C_{1}(a, t) =& I(1, a, t) - \frac{G(a)}{2G(a) - G(t)}\Big[G(a) + H(a)\Psi\left(\frac{a}{t}\right)\left(1 - \frac{G(t)}{2G(a) - G(t)}\right)\Big], \\
C_{2}(a, t) =&  \frac{H(a)\Psi\left(\frac{a}{t}\right)}{2G(a) - G(t)}, \\
C_{3}(a, t) =& \frac{H(a) \Psi\left(\frac{a}{t}\right) +  G(a)}{2G(a) - G(t)} - \frac{2H(a)G(a)\Psi\left(\frac{a}{t}\right)}{(2G(a) - G(t))^2}.
\end{cases}
\end{eqnarray*}
For simplicity of notation, we will write $C_{i}$ for $C_{i}(a, t), i = 1, 2, 3$.
\end{thm}
\begin{proof}
First, by Lemma \ref{lem:psiat} we find that
$$
I(1, a, t) = \int_{-\infty}^{a}\Psi\left(\frac{x_1}{t}\right)\nu(dx_1) = \frac{G(t)}{2}  + \frac{1}{2} H(a) \Psi\left(\frac{a}{t}\right) + \frac{G(a)}{2}.
$$
By the same Lemma we have
\begin{align*}
    I(n, a, t) &=& \frac{G(t)}{2}I(n-1, a, t)  + \frac{M(a, t)}{2} I(n-1, a, a) + \frac{G(a)}{2} \Psi\left(\frac{a}{t}\right)  II(n-1, a, a).
    \end{align*}
Iterating this equation, we can see that 
\begin{eqnarray*}
& & I(n, a, t) = \left(\frac{G(t)}{2}\right)^{n - 1}I(1, a, t) \\
&+& \frac{M(a, t)}{2}\sum_{k = 1}^{n - 1}I(k, a, a)\left(\frac{G(t)}{2}\right)^{n - 1 - k} + \frac{G(a)\Psi\left(\frac{a}{t}\right)}{2}\sum_{k = 1}^{n - 1}II(k, a, a)\left(\frac{G(t)}{2}\right)^{n - 1 - k}.
\end{eqnarray*}
To check this equality, we find that
\begin{eqnarray*}
& &\frac{G(t)}{2}I(n-1, a, t)  + \frac{M(a, t)}{2} I(n-1, a, a) + \frac{G(a)\Psi\left(\frac{a}{t}\right)}{2}   II(n-1, a, a) \\
&=& \left(\frac{G(t)}{2}\right)^{n - 1}I(1, a, t) + \frac{M(a, t)}{2}\sum_{k = 1}^{n - 2}I(k, a, a)\left(\frac{G(t)}{2}\right)^{n - 1 - k} \\
&+& \frac{G(a)\Psi\left(\frac{a}{t}\right)}{2}\sum_{k = 1}^{n - 2}II(k, a, a)\left(\frac{G(t)}{2}\right)^{n - 1 - k} + \frac{M(a, t)}{2} I(n-1, a, a) \\
&+& \frac{G(a)}{2} \Psi\left(\frac{a}{t}\right)  II(n-1, a, a) \\
&=& \left(\frac{G(t)}{2}\right)^{n - 1}I(1, a, t) + \frac{M(a, t)}{2}\sum_{k = 1}^{n - 1}I(k, a, a)\left(\frac{G(t)}{2}\right)^{n - 1 - k} \\
&+& \frac{G(a)\Psi\left(\frac{a}{t}\right)}{2}\sum_{k = 1}^{n - 1}II(k, a, a)\left(\frac{G(t)}{2}\right)^{n - 1 - k} = I(n, a, t).
\end{eqnarray*}
It is enough to find the two sums which we used in the above formula.
Using simple algebra we find that 
\begin{eqnarray*}
& & \sum_{k = 1}^{n - 1}I(k, a, a)\left(\frac{G(t)}{2}\right)^{n - 1 - k} = \sum_{k = 1}^{n - 1}G(a)^{k}\left(\frac{G(t)}{2}\right)^{n - 1 - k} \\
&=& \frac{2G(a)}{2G(a) - G(t)}\Bigg[G(a)^{n - 1} - \left(\frac{G(t)}{2}\right)^{n - 1}\Bigg]
\end{eqnarray*}
and
\begin{eqnarray*}
& & \sum_{k = 1}^{n - 1}II(k, a, a)\left(\frac{G(t)}{2}\right)^{n - 1 - k} = \sum_{k = 1}^{n - 1}G(a)^{k - 1}\left[kH(a) + G(a)\right]\left(\frac{G(t)}{2}\right)^{n - 1 - k} \\
&=& 2G(a)^{n-1} \left[ \frac{ nH(a) + G(a)}{\left(2G(a)-G(t)\right)} - \frac{2G(a)H(a)}{\left(2G(a)-G(t)\right)^2}\right]\\
&-& 2\left(\frac{G(t)}{2}\right)^{n-1} \left[ \frac{  G(a)}{\left(2G(a)-G(t)\right)} - \frac{G(t)H(a)}{\left(2
G(a)-G(t)\right)^2}\right].
\end{eqnarray*}
Combining these results ends the proof. 
\end{proof}

Now we can find the formula for $II(n, a, t)$.
\begin{thm}\label{thm:iinat}
The integral $II(n, a, t)$ is given by
\begin{eqnarray*}
II(n, a, t) &=&  G(a)^{n}\Bigg[\frac{(n + 1)H(a) + G(a)}{2G(a) - G(t)} - \frac{2G(a)H(a)}{(2G(a) - G(t))^2} \\
&+& \frac{H(t)}{2G(a) - G(t)}\left(nC_2 + C_3 - \frac{2C_{2}G(a)}{2G(a) - G(t)}\right)\Bigg] \\
&+& \left(\frac{G(t)}{2}\right)^{n-1}\Bigg[II(1, a, t) - \frac{G(a)(H(a) + G(a))}{2G(a) - G(t)} + \frac{G(a)G(t)H(a)}{(2G(a) - G(t))^2} \\
&+& \frac{(n - 1)C_{1}H(t)}{G(t)} - \frac{G(a)H(t)}{2G(a) - G(t)}\left(C_3 - \frac{C_{2}G(t)}{2G(a) - G(t)}\right)\Bigg].
\end{eqnarray*}
for $n \geq 2$ and $G(t) \neq 2G(a)$.
\end{thm}
\begin{proof}
First, let us notice that $G(t) > 0$, since $t > 0$.
By Lemma \ref{lem:2} and the definition of $H(t)$ we have 
$$
II(1, a, t) = \frac{H(a)+H(t)+G(a)+G(t)}{2}.
$$
By the same Lemma we find that
\begin{eqnarray*}
 II(n, a, t) &=& \frac{H(a)}{2}I(n-1, a, a)  + \frac{G(a)}{2} II(n-1, a, a)\\
 &+& \frac{H(t)}{2}I(n-1, a, t) + \frac{G(t)}{2} II(n-1, a, t).  
\end{eqnarray*}
By iterating the above formula for $II(n, a, t)$ we can see that
\begin{eqnarray*}
& &  II(n, a, t) = II(1, a, t) \left(\frac{G(t)}{2}\right)^{n-1} + \frac{H(a)}{2} \sum\limits_{k=1}^{n-1} I(k, a, a) \left(\frac{G(t)}{2}\right)^{n - 1 - k} \\
&+& \frac{G(a)}{2}  \sum\limits_{k = 1}^{n - 1} II(k, a, a) \left(\frac{G(t)}{2}\right)^{n - 1 - k} + \frac{H(t)}{2}  \sum\limits_{k = 1}^{n - 1} I(k, a, t) \left(\frac{G(t)}{2}\right)^{n - 1 - k}.
\end{eqnarray*}
An argument similar to the one provided for $I(n, a, t)$ convinces us that this expression solves the recurrence equation that defines $II(n, a, t)$.
It is sufficient to find closed form of the sum $\sum\limits_{k=1}^{n-1} I(k, a, t) \left(\frac{G(t)}{2}\right)^{n - 1 - k}$.
We have
\begin{eqnarray*}
& &\sum\limits_{k=1}^{n-1} I(k, a, t) \left(\frac{G(t)}{2}\right)^{n - 1 - k} \\
&=& \frac{2G(a)^n}{2G(a) - G(t)} \left(nC_{2} + C_{3} - \frac{2C_{2} G(a)}{2G(a) - G(t)}\right) \\
&+& \left(\frac{G(t)}{2}\right)^{n - 2}\left[(n - 1)C_{1} - \frac{G(a)G(t)}{(2G(a) - G(t))}\left(C_{3} - \frac{C_{2}G(t)}{2G(a) - G(t)}\right)\right].
\end{eqnarray*}
A simple use of algebra ends the proof.
\end{proof}

For both $I(n, a, t)$ and $II(n, a, t)$ we needed to assume that $G(t) \neq 2G(a)$. In case $G(t) = 2G(a)$ it is easy to check that we have
\small{
\begin{eqnarray*}
& & I(n, a, t) = G(a)^{n - 1}\Bigg[I(1, a, t) + \frac{n - 1}{2}\left(B + n\frac{H(a)\Psi\left(\frac{a}{t}\right)}{2}\right)\Bigg],\\
& & II(n, a, t) = G(a)^{n - 1}\Bigg[II(1, a, t) + \frac{H(a)(n - 1)}{2}\left(1 + \frac{n}{2}\right) + \frac{G(a)(n - 1)}{2} \\
&+& \frac{H(t)}{2G(a)}\Big((n - 1)I(1, a, t) + \frac{B(n - 1)(n - 2)}{4} + \frac{H(a)\Psi\left(\frac{a}{t}\right)n(n - 1)(n - 2)}{12}\Big)\Bigg],
\end{eqnarray*}}
$B = H(a)\Psi\left(\frac{a}{t}\right) + G(a)$.

\section{Fluctuations of Kendall random walk}

For any positive $a$ we introduce the first hitting times of the half lines $[a,\infty)$ and $(-\infty, a]$ for the random walk $\{X_n: n\in\mathbb{N}_0\}$:
$$
\tau_a^+=\min\{n\geqslant 1: X_n>a\}, \quad \tau_a^-=\min\{n\geqslant 1: X_n<a\}
$$
and weak ascending and descending ladder variables:
$$
\tpw=\min\{n\geqslant 1: X_n \geq a\}, \quad \tmw=\min\{n\geqslant 1: X_n \leq a\}
$$
with  convention $\min\emptyset = \infty$. In \cite{factor} authors found joint distribution of the random vector $(\tau_0^+, X_{\tau_0^+})$. Our main goal here is to extend this result for any $a \geq 0$. 

\begin{lem}\label{lem:3}
Random variable $\tau_0^+$ (and, by symmetry of the Kendall random walk, also variable $\tau_0^-$) has geometric distribution 
$$
P(\tau_0^{+} = k) = \frac{1}{2^k}, \quad k=1,2,\cdots.
$$
\end{lem}

We will now investigate the distribution of the random variable $\tp$.
First, we notice that 
\begin{eqnarray*}
\mathbb{P}(\tp = n) & = & \gp(X_{0} \leq a, X_{1} \leq a, \ldots, X_{n - 1} \leq a, X_{n} > a) \\
&& \hspace{-15mm} = \cd\ldots\cd{\cg}P_{1}(x_{n - 1}, dx_{n})P_{1}(x_{n - 2}, dx_{n - 1}){\ldots}P_{1}(0, dx_{1})
\end{eqnarray*}

At the beginning, we will compute the value of the innermost integral. 
The result is given in the following lemma.
\begin{lem}\label{lem:7}
$$
{\cg}P_{1}(x_{n - 1}, dx_{n}) = \frac{1}{2} - \frac{1}{2}\left[\Psi\left(\frac{x_{n - 1}}{a}\right)H(a) + G(a)\right]\mathbf{1}_{(|x_{n - 1}| < a)}
$$
where
$$
H(a) = 2F(a) - 1 - G(a).
$$
\end{lem}
\noindent {\bf Proof.} 
\begin{eqnarray*}
{\cg}P_{1}(x_{n - 1}, dx_{n}) & = & \cg(\delta_{x_{n - 1}}\sk\nu)(dx_{n}) \\
& = & (\delta_{x_{n - 1}}\sk\nu)(a, \infty) \\
& = & \frac{1}{2} - \frac{1}{2}\left[\Psi\left(\frac{x_{n - 1}}{a}\right)(2F(a) - 1 - G(a)) + G(a)\right]\mathbf{1}_{(|x_{n - 1}| < a)}
\end{eqnarray*}
by the symmetry of $\delta_{x_{n - 1}}\sk\nu$ measure and by Lemma \ref{lem:2}.

Iterating Lemma \ref{lem:psiat} $n$ times we arrive at the tail distribution of $\tmw$:
\begin{lem}\label{lem:8}
\begin{eqnarray*}
\mathbb{P}(\tmw > n + 1) = \frac{1}{2^{n}} \left(1-F(a)\right)
\end{eqnarray*}
and 
\begin{eqnarray*}
\mathbb{P}(\tmw = 1) = F(a).
\end{eqnarray*}
\end{lem}

\noindent {\bf Proof.} Indeed
\begin{eqnarray*}
\mathbb{P}(\tmw > n + 1) & = & \gp(X_{0} > a, X_{1} > a, \ldots, X_{n} > a, X_{n + 1} > a) \\
&& \hspace{-15mm} = \cg\ldots\cg{\cg}P_{1}(x_{n}, dx_{n + 1})P_{1}(x_{n - 1}, dx_{n}){\ldots}P_{1}(0, dx_{1}).
\end{eqnarray*}
Since the inner integral is given by Lemma \ref{lem:7} by
\begin{eqnarray*}
{\cg}P_{1}(x_{n}, dx_{n + 1}) & = & \left(\delta_{x_{n}}\sk\nu\right)(a, \infty) \\
&& \hspace{-15mm} = \frac{1}{2} - \frac{1}{2}\left[\Psi\left(\frac{x_{n}}{a}\right)H(a) + G(a)\right]\mathbf{1}_{(|x_{n}| < a)}
\end{eqnarray*}
and the second factor of the above expression is equal zero for $x_n > a$, we see that 
\begin{eqnarray*}
{\cg}{\cg}P_{1}(x_{n}, dx_{n + 1}) P_{1}(x_{n - 1}, dx_{n})= \frac{1}{2} \left(\delta_{x_{n - 1}}\sk\nu\right)(a, \infty) .
\end{eqnarray*}
Hence 
\begin{eqnarray*}
\cg\ldots\cg{\cg}P_{1}(x_{n}, dx_{n + 1})P_{1}(x_{n - 1}, dx_{n}){\ldots}P_{1}(0, dx_{1}) = \frac{1}{2^n} \left(\delta_{0}\sk\nu\right)(a, \infty),
\end{eqnarray*}
which ends the proof of the first formula because $\delta_{0}\sk\nu = \nu$.

Moreover 
\begin{eqnarray*}
\mathbb{P}(\tmw = 1) = \lim\limits_{n\to 0} \mathbb{P}(\tmw \leq n + 1) = F(a), 
\end{eqnarray*}
i.e. distribution of $\tmw$ has atom at 1 with mass $F(a)$ and 
\begin{eqnarray*}
\mathbb{P}(\tmw = k) = \frac{1}{2^{n - 1}} \left(1-F(a)\right) 
\end{eqnarray*}
for $n \geq 2$. It exactly means that $\tmw$ has geometrical distribution for $n \geq 2$ with mass:
\begin{eqnarray*}
\mathbb{P}(\tmw > 1) = 1- F(a). 
\end{eqnarray*}
\qed

Now it is evident that distribution of $\tmw$ depends only on cumulative distribution function of the unit step and distribution of $\tau_0^+$:
\begin{cor}
$$
\mathbb{P}(\tmw = n) =  
\left\{\begin{array}{ll}
  F(a) & if \quad n = 1, \\
  \left(1-F(a)\right) \mathbb{P}(\tau_0^+ = n - 1) & if \quad n \geq 2.
            \end{array}\right.
$$
\end{cor}

Let $\tau_{a}^{+}$ denote the first ladder moment for the Kendall random walk, meaning that
$$
\tau_{a}^{+} = \inf\{n \geq 1: X_{n} > a\},
$$
where $(X_{n})$ is the Kendall random walk.
In this section, we prove the first important result of this paper: formula for the probability distribution function of the random variable $\tau_{a}^{+}$.
\begin{thm}\label{thm:drabina}
For any $a \geq 0$ and $n \in \mathbb{N}$
\begin{eqnarray*} 
\mathbb{P}(\tau_{a}^{+} = n) &=& A(a) \mathbb{}\left(\frac{1}{2}\right)^{n} + B(a) n(1 - G(a))^{2}G(a)^{n - 1} + C(a) G(a)^{n - 1}(1 - G(a)),
\end{eqnarray*}
where
\begin{eqnarray*}
\begin{cases}
A(a) &= 1 + \frac{H(a)}{(2G(a) - 1)^{2}}  - \frac{G(a)}{(2G(a) - 1)} \\
B(a) &= \frac{H(a)}{(2G(a) - 1)(1 - G(a))} \\
C(a) &= \frac{G(a)}{2G(a) - 1} - \frac{H(a)}{(2G(a) - 1)^{2}}\frac{G(a)}{(1 - G(a))}.
\end{cases}
\end{eqnarray*}
\end{thm}
It is easy to see that $A(a) + B(a) + C(a) = 1$.
The distribution of $\tau_{a}^{+}$ is a convex combination of two geometric distributions, one with a probability of success equal to $\frac{1}{2}$, one with a probability of success equal to G(a) and a shifted negative binomial distribution with parameters 2 and G(a). Thus, it is a mixture of negative binomial distributions with coefficients that depend on the unit step distribution of the associated Kendall random walk both through its CDF and Williamson transform.

Proof of this formula uses Markov property of the Kendall random walk. 
The probability distribution function of $\tau_{a}^{+}$ is expressed as an iterated integral with respect to the transition kernels.
Results of consecutive integrations form a sequence.
We will find its closed form to calculate $\mathbb{P}(\tau_{a}^{+} = n)$.

\begin{proof} 
Since $X_{n}$ is a Markov process, we have
\begin{eqnarray*}
\mathbb{P}(\tau_{a}^{+} = n) & = & \mathbb{P}(X_{0} \leq a, X_{1} \leq a, \ldots, X_{n - 1} \leq a, X_{n} > a) \\
&=&\int_{-\infty}^{a}\ldots\int_{-\infty}^{a}{\int_{a}^{\infty}}P_{1}(x_{n - 1}, dx_{n})P_{1}(x_{n - 2}, dx_{n - 1}){\ldots}P_{1}(0, dx_{1}).
\end{eqnarray*}
Using Lemma \ref{lem:2}, we can compute the innermost integral
\begin{eqnarray*}
\int_{a}^{\infty}P_{1}(x_{n - 1}, dx_{n}) = \frac{1}{2} - \frac{1}{2}\left[\Psi\left(\frac{x_{n - 1}}{a}\right)H(a) + G(a)\right]\mathbf{1}_{(|x_{n - 1}| < a)}
\end{eqnarray*}
Now we can define a sequence $(I_{j})$, where $I_{j}$ denotes the result of the $j$-th integration.
Let
$$
I_{1} = {\cg}P_{1}(x_{n - 1}, dx_{n})
$$ 
and
\[
I_{j + 1} = {\cd}I_{j} \left(\delta_{x_{n - j - 1}}\sk\nu\right)(d_{x_{n - j}}).
\]
Let us notice that $I_{j}$ is of the form 
$$
A_{j} + {\Psi}\left(\frac{x_{n - j}}{a}\right)H(a)B_{j} + C_{j}G(a)\mathbbm{1}_{(|x_{n - j}| < a)},
$$
which is easy to verify by integrating this formula with respect to $x_{n - j}$.
This way we also obtain recurrence equations for $A_{j}$, $B_{j}$ and $C_{j}$ sequences.
We have
$$
\begin{cases}
A_{j + 1} =& \frac{1}{2}A_{j}, \\
B_{j + 1} =& \frac{1}{2}A_{j} + G(a)(B_{j} + C_{j}), \\
C_{j + 1} =& \frac{1}{2}A_{j} + C_{j}G(a) 
\end{cases}
$$  
with initial conditions $A_{1} = \frac{1}{2}$, $B_{1} = C_{1} = -\frac{1}{2}$.
It is easy to see that $A_{j}$ sequence is a geometric sequence.
After plugging the formula for $A_{j}$ into the equations that define $B_{j}$ and $C_{j}$ and iterating the formulas for $B_{j}$ and $C_{j}$, we arrive at the following  solutions
$$
\begin{cases}
  	B_{j} =& G(a)^{j- 1}\frac{j(1 - G(a))(2G(a) - 1) - G(a)}{(2G(a) - 1)^{2}} + \left(\frac{1}{2}\right)^{j}\frac{1}{(2G(a) - 1)^{2}},\\
  	C_{j} =& \frac{G(a)^{j - 1}(1 - G(a)) - 2^{-j}}{2G(a) - 1} = G(a)^{j - 1}\frac{1 - G(a)}{2G(a) - 1} - \left(\frac{1}{2}\right)^{j}\frac{1}{2G(a) - 1}.
  \end{cases} 	
$$
We will check that these sequences satisfy our recurrence equations.
For sequence $B_{j}$ 
\begin{eqnarray*}
& &\left(\frac{1}{2}\right)^{j + 1} + G(a)(B_{j} + C_{j}) \\
&=& G(a)^{j}\bigg[\frac{j(1 - G(a))(2G(a) - 1) - G(a)}{(2G(a) - 1)^{2}} + \frac{1 - G(a)}{2G(a) - 1}\bigg] \\
&+& \left(\frac{1}{2}\right)^{j + 1}\bigg[1 + \frac{2G(a)}{(2G(a) - 1)^{2}} - \frac{2G(a)}{2G(a) - 1}\bigg] \\
&=& G(a)^{j}\bigg[\frac{(j + 1)(1 - G(a))(2G(a) - 1) - G(a)}{(2G(a) - 1)^{2}}\bigg] \\
&+& \left(\frac{1}{2}\right)^{j + 1}\frac{1}{(2G(a) - 1)^{2}} = B_{j + 1}
\end{eqnarray*}
For sequence $C_{j}$ we have
\begin{eqnarray*} 
G(a)C_{j} + 2^{-(j + 1)} &=& \frac{G(a)^{j}(1 - G(a)) - 2^{-j}G(a)}{2G(a) - 1} + 2^{-(j + 1)} \\
&=& \frac{G(a)^{j}(1 - G(a)) - 2^{-j}G(a) + 2^{-(j + 1)}(2G(a) - 1)}{2G(a) - 1} \\
&=& \frac{G(a)^{j}(1 - G(a))  - 2^{-(j + 1)}}{2G(a) - 1} = C_{j + 1}
 \end{eqnarray*}
Since $\mathbbm{1}(|x_{0}| < a) = 1$ and substituting $x_{0} = 0$ implies $\Psi\left(\frac{x_{0}}{t}\right) = 1$, we have proven the formula for the probability distribution function of $\tau_{a}^{+}$, which can be seen after we group terms in formulas for sequences $B_{j}$ and $C_{j}$.
\end{proof}

\section{Distribution of the first ladder height}

In this section, we give the formula for the cumulative distribution function of the  first ladder  height over any level $a \geq 0$. At the beginning let us look on a special case of the desired result in th $a=0$, which was proved in \cite{factor}.

\begin{thm}\label{thm:wynikazero}
Cumulative distribution function of the joint distribution of random variables $\tau_{0}^{+}$ and $X_{\tau_{0}^{+}}$ is given by  
\begin{eqnarray*}
\Phi^0_{n}(t) = \mathbb{P}(X_{\tau_{0}^{+}} \leq t, \tau_{0}^{+} = n) = \frac{1}{2^{n}}G(t)^{n - 1}\Big[2n\left(F(t) - \frac{1}{2}\right) - (n - 1)G(t)\Big].
\end{eqnarray*}
\end{thm}
Notice that
 \begin{eqnarray*}
\Phi^0_n(t) = \mathbf{P} \bigl\{ \tau_0^{+} = n\bigr\} \mathbf{P} \bigl\{|X_n|< t\bigr\}
 \end{eqnarray*}
and 
$$
\mathbb{P}\left(X_{\tau_{0}^{+}} \leq t\right) = \frac{4F(t) - 2 - G(t)^{2}}{(2 - G(t))^{2}}.
$$
Our goal in this section is to generalize this result for any level $a \geq 0$ in the following way:
\begin{thm}
Cumulative distribution function of the joint distribution of random variables $\tau_{a}^{+}$ and $X_{\tau_{a}^{+}}$ is given by  
\begin{eqnarray*}
& & \Phi^a_{n}(t) := \mathbb{P}(X_{\tau_{a}^{+}} \leq t, \tau_{a}^{+} = n) = \\
&=& \left(\frac{G(t)}{2}\right)^{n-1}\Bigg[\frac{2G(a)H(a)(G(t) - G(a))}{(2G(a) - G(t))^2}- \frac{G(a)^{2}}{2G(a) - G(t)} \\
&+& \frac{(n - 1)C_{1}H(t)}{G(t)} - \frac{G(a)H(t)}{2G(a) - G(t)}\left(C_{3} - \frac{C_{2}G(t)}{2G(a) - G(t)}\right) + II(1, a, t)\Bigg] \\
&+& G(a)^{n - 1}\Bigg[\frac{(nH(a) + G(a))(G(t) - G(a))}{2G(a) - G(t)} - \frac{G(a)G(t)H(a)}{(2G(a) - G(t))^2} \\
&+& \frac{G(a)H(t)C_{2}n}{2G(a) - G(t)} + \frac{G(t)H(t)}{2G(a) - G(t))}\left(\frac{C_3 - C_2}{2} - \frac{G(a)C_2}{2G(a) - G(t)}\right) + \frac{H(t)(C_{3} - C_{2})}{2}\Bigg]
\end{eqnarray*}
for $n \geq 2$, $t > a$ such that $|G(t)| < 1$ and expressions $C_1, C_2, C_3$ defined in Theorem \ref{thm:inat}.
Moreover $\Phi^a_{n}(0)=0$. For $n = 1$ we have simply $\mathbb{P}(X_{\tau_{a}^{+}} \leq t, \tau_{a}^{+} = 1) = F(t) - F(a)$. 
Since $G(0) = 0$ and consequently $H(0) = 0$, it is easy to see that for $a = 0$ this expression simplifies to the expression given in Theorem \ref{thm:wynikazero}.
We also need the technical assumption that $G(t) \neq 2G(a)$.
\end{thm}

\begin{proof}
Let us introduce the notation
\begin{eqnarray*}
\Phi^a_n(t) & := &  \mathbf{P} \bigl\{ \tau_a^{+} = n,  X_{\tau_a^{+}} < t\bigr\} =  \mathbb{P}(X_{1} \leq a, \ldots, X_{n - 1} \leq a, a < X_{n} \leq t) \\
&=& \int_{-\infty}^{a}\ldots\int_{-\infty}^{a}\int_{a}^{t}(\delta_{x_{n - 1}}\vartriangle_{\alpha}\nu)(dx_{n})(\delta_{x_{n - 2}}\vartriangle_{\alpha}\nu)(dx_{n - 1})\ldots\nu(dx_{1}).
\end{eqnarray*}
The innermost integral can be computed using Lemma \ref{lem:2}.
Then this probability can be expressed as 
$$
\Phi_{n}^{a}(t) = \frac{H(t)}{2}I(n - 1, a, t) + \frac{G(t)}{2}II(n - 1, a, t) - \frac{H(a)}{2}I(n - 1, a, a) - \frac{G(a)}{2}II(n - 1, a, a).
$$
Substituting the expressions obtained in Theorems \ref{thm:inat} and \ref{thm:iinat} for the terms $I(n - 1, a, t)$, $II(n - 1, a , t)$ $I(n - 1, a, a)$ and $II(n - 1, a, a)$ ends the proof.
\end{proof}

We can now find the marginal distribution of the random variable $X_{\tau_{a}^{+}}$.
\begin{cor}
Cumulative distribution function of the random variable $X_{\tau_{a}^{+}}$ is given by the following formula.
\small{
\begin{eqnarray*}
& & \mathbb{P}(X_{\tau_{a}^{+}} \leq t) = F(t) - F(a) \\
&+& \frac{G(t)}{2 - G(t)}\Bigg[\frac{H(t)(4 - G(t))C_1}{G(t)(2 - G(t))} + \frac{2G(a)H(a)(G(t) - G(a))}{(2G(a) - G(t))^2} \\
&-& \frac{G(a)^{2}}{2G(a) - G(t)} - \frac{C_{1}H(t)}{G(t)} -\frac{G(a)H(t)}{2G(a) - G(t)}\left(C_{3} - \frac{C_{2}G(t)}{2G(a) - G(t)}\right) + II(1, a, t)\Bigg] \\
&+& \frac{G(a)}{1 - G(a)}\Bigg[\frac{(2 - G(a))(H(a)(G(t) - G(a)) + G(a)H(t)C_2)}{(1 - G(a))(2G(a) - G(t))} + \frac{G(a)(G(t) - G(a))}{2G(a) - G(t)} \\
&-& \frac{G(a)G(t)H(a)}{(2G(a) - G(t))^2} +  \frac{G(t)H(t)}{2G(a) - G(t)}\left(\frac{C_3 - C_2}{2} - \frac{G(a)C_2}{2G(a) - G(t)}\right) + \frac{H(t)(C_{3} - C_{2})}{2}\Bigg].
\end{eqnarray*}}
\end{cor}
Again, it is easy to check that for $a = 0$ this expression simplifies to the expression given in Theorem \ref{thm:wynikazero}.

\subsection{Maxima and minima of Kendall random walks}

In this section we will prove an analog of  Pollaczek-Khintchine formula. We will start with a lemma that describes the distribution of the maximum of $n$ steps of a Kendall random walk.

\begin{lem}\label{lem:14}
Let $\{X_n: n\in\mathbb{N}_0\}$ denote the Kendall random walk. Then the distribution of $\max\limits_{0 \leq i \leq n}X_{i}$ is given by
\begin{eqnarray*}
& & \mathbb{P}(\max\limits_{0 \leq i \leq n}X_{i} \leq t) \\
& & = A(t)\mathbb{P}(\tau_{0}^{+} = n) + B(t)\frac{G(t)}{1 - G(t)}(1 - G(t))^{2}nG(t)^{n - 1} \\
& & + (B(t) + C(t)) \frac{G(t)}{1 - G(t)}G(t)^{n - 1}(1 - G(t))
\end{eqnarray*}
for functions $A$, $B$ and $C$ defined in Theorem \ref{thm:drabina} and $t > 0$.
\end{lem}

\noindent
{\bf Proof.} It is sufficient to see that
\begin{eqnarray*}
& &\mathbb{P}(\max\limits_{0 \leq i \leq n}X_{i} \leq t) = \mathbb{P}(X_{1} \leq t, \ldots, X_{n} \leq t) \\
&=& \mathbb{P}(X_{1} \leq t, \ldots, X_{n} \leq t, X_{n + 1} \leq t) \\
&+& \mathbb{P}(X_{1} \leq t, \ldots, X_{n} \leq t, X_{n + 1} > t) \\
&=& \mathbb{P}(\tau_{t}^{+} > n + 1) + \mathbb{P}(\tau_{t}^{+} = n + 1) = \mathbb{P}(\tau_{t}^{+} \geq n + 1).
\end{eqnarray*}
Summation of the formula for the distribution of $\tau_{t}^{+}$ ends the proof.
\qed
\begin{lem}\label{lem:15}
Let $\{X_n: n\in\mathbb{N}_0\}$ denote the Kendall random walk. Then the distribution of $\min\limits_{0 \leq i \leq n}X_{i}$ is given by
\begin{eqnarray*}
\mathbb{P}(\min\limits_{0 \leq i \leq n}X_{i} \leq t ) &=& 1 - \left(\frac{1}{2}\right)^{n}\bigg[1 + \frac{H(-t)}{(2G(t) - 1)^{2}} - \frac{G(t)}{2G(t) - 1}\bigg] \\
&-& G(t)^{n}\bigg[\frac{G(t)}{2G(t) - 1}- \frac{2H(-t)}{(2G(t) - 1)^{2}}\bigg] \\
&-& nG(t)^{n}\frac{2H(-t)G(t)}{(2G(t) - 1)^{2}}
\end{eqnarray*}
for $t < 0$.
\end{lem}
\noindent
{\bf Proof.} 
The proof is a simple modification of the proof of Theorem \ref{thm:drabina}. For $t < 0$ we have 
\begin{eqnarray*} 
& & \mathbb{P}(\min_{0 \leq i \leq n}X_{i} \leq t) = 1 - \mathbb{P}(\min_{0 \leq i \leq n}X_{i} > t) \\
&=& 1 - \mathbb{P}(X_{0} > t, X_{1} > t, \ldots, X_{n} > t) \\
&=& 1 - \int_{t}^{\infty}\ldots\int_{t}^{\infty}(\delta_{x_{n - 1}}\sk\nu)(dx_{n})\ldots(\delta_{x_{0}}\sk\nu)(dx_{1}).
\end{eqnarray*}
We define 
$$
I_{1} := \int_{t}^{\infty}(\delta_{x_{n - 1}}\sk\nu)(dx_{n})
$$ 
and recursively
$$
I_{j + 1} := \int_{t}^{\infty}(\delta_{x_{n - j - 1}}\sk\nu)(dx_{n - j})
$$
for $1 \leq j \leq n-1$ and $x_0=0$.
Assuming  $I_{j} = A_{j} + B_{j}\Psi\left(\frac{x_{n - j}}{t}\right) + C_{j}\mathbbm{1}_{(|x_{n - j}| < -t)}$ we have
\begin{eqnarray*}
& &I_{j + 1} = \int_{t}^{\infty}\bigg[A_{j} + B_{j}\Psi\left(\frac{x_{n - j}}{t}\right) + C_{j}\mathbbm{1}(|x_{n - j}| < -t)\bigg](\delta_{x_{n - j - 1}}\sk\nu)(dx_{n - j}) \\
&=& A_{j}\int_{t}^{\infty}(\delta_{x_{n - j - 1}}\sk\nu)(dx_{n - j}) \\
&+& B_{j}\int_{t}^{\infty}\Psi\left(\frac{x_{n - j}}{t}\right)(\delta_{x_{n - j - 1}}\sk\nu)(dx_{n - j}) \\
&+& C_{j}\int_{t}^{-t}(\delta_{x_{n - j - 1}}\sk\nu)(dx_{n - j}) \\
&=& A_{j}\bigg[\frac{1}{2} + \frac{1}{2}\left(H(-t)\Psi\left(\frac{x_{n - j - 1}}{t}\right) + G(t)\right)\mathbbm{1}_{(|x_{n - j - 1}| < -t)}\bigg] \\
&+& B_{j}G(t)\Psi\left(\frac{x_{n - j - 1}}{t}\right) + C_{j}\left(H(-t)\Psi\left(\frac{x_{n - j - 1}}{t}\right) + G(t)\right)\mathbbm{1}_{(|x_{n - j - 1}| < -t)} \\
&=& \frac{1}{2}A_{j} + \Psi\left(\frac{x_{n - j - 1}}{t}\right)\bigg[\frac{1}{2}H(-t)A_{j} + G(t)B_{j} + C_{j}H(-t)\bigg] \\
&+& \mathbbm{1}_{(|x_{n - j - 1}| < -t)}\bigg[\frac{1}{2}G(t)A_{j} + G(t)C_{j}\bigg].
\end{eqnarray*}
Thus we arrive at the following set of recurrence equations.
\begin{eqnarray*}
\begin{cases}
A_{j + 1} &= \frac{1}{2}A_{j}, \\
B_{j + 1} &= G(t)B_{j} + H(-t)(A_{j + 1} + C_{j}), \\
C_{j + 1} &= G(t)A_{j + 1} + G(t)C_{j}
\end{cases}
\end{eqnarray*}
with the initial conditions $A_{1} = \frac{1}{2}$, $B_{1} = \frac{1}{2}H(-t)$, $C_{1} = \frac{1}{2}G(t)$.
It is easy to check that the solutions are given by the following sequences
\begin{eqnarray*}
\begin{cases}
A_{j} &= \frac{1}{2^{j}}, \\
B_{j} &= jG(t)^{j}\bigg[\frac{2G(t)H(-t)}{(2G(t) - 1)^{2}}\bigg] - G(t)^{j}\frac{2H(-t)}{(2G(t) - 1)^{2}} + 2^{-j}\bigg[\frac{H(-t)}{(2G(t) - 1)^{2}}\bigg], \\
C_{j=} &= \frac{G(t)}{2G(t) - 1}\bigg[G(t)^{j} - \frac{1}{2^{j}}\bigg].
\end{cases}
\end{eqnarray*}
\qed

\section{Regular variation approach to the ladder points}
The expressions of cumulative distribution functions of the coordinates of the ladder points are complicated. Our main goal in this section is to investigate asymptotic properties of cdf using regularly varying functions at infinity (for a survey of regularly varying functions and their applications we refer to \cite{BGT}, \cite{Seneta}).

\begin{defn}
A positive and measurable function $f$ is regularly varying at infinity and
with index $\beta $ (notation $f\in RV_{\beta }$) if it satisfies
\[
\lim_{t\rightarrow \infty }\frac{f(tx)}{f(t)}=x^{\beta },\forall x>0\text{.}
\]
\end{defn}
In \cite{renewalKendall} one can find asymptotic properties for the Williamson transform of the unit step and corresponding characteristics. Here we consider symmetric random walks. 

Let $X$ be a symmetric random variable around $0$ with distribution $\nu \in \mathcal{P}_s$, where $P\left\{ X=0\right\} =0$. The Williamson transform of probability measure $\nu \in \mathcal{P}_s$ is given as follows
\[
G(t)=E\left(1-\left\vert \frac{X}{t}\right\vert ^{\alpha }\right)_{+}\text{,} 
\]%
where $a_{+}=\max (a,0)$.

Clearly we have $G(t)=G(-t)=G(\left\vert t\right\vert )$ and we
choose $t>0$ from now on.

\bigskip

For convenience we set $Y=\left\vert X\right\vert $ and $G(t)=G_{Y}(t)$.

Moreover $F_{Y}(0)=0$ and 
\[
F_{Y}(x)=2F(x)-1,x\geq 0 
\]

Conversely, for $x\geq 0$ we have%
\begin{eqnarray*}
F(x) &=&\frac{1+F_{Y}(x)}{2}, \\
F(-x) &=&1-F(x)=\frac{1-F_{Y}(x)}{2}.
\end{eqnarray*}
Let
$$
W_Y(x) := \int\limits_0^x t^{\alpha-1} \overline{F}_Y(t) dt.
$$
As to $G_Y(x)$ note that we have
$$
G_{Y}(x)=\alpha x^{-\alpha }\int_{0}^{x}t^{\alpha -1}F_{Y}(t)dt. 
$$
Using the notations as above, we have
\begin{eqnarray*}
F_{Y}(x) &=&G_{Y}(x)+\frac{x}{\alpha }G_{Y}^{\prime }(x), \\
G_{Y}(x) &=&F_{Y}(x)-H_{Y}(x), \\
\alpha W_{Y}(x) &=& x^{\alpha} H_{Y}(x)+x^{\alpha }\overline{F}_{Y}(x).
\end{eqnarray*}

The first formula gives the inverse of the Williamson transform. In terms of 
$X$, for $x>0$ we find%
\begin{eqnarray*}
F(x) &=&\frac{1}{2}(1+G_{Y}(x)+\frac{t}{\alpha }G_{Y}^{\prime }(x))\text{
,} \\
F(-x) &=&1-F(x).
\end{eqnarray*}
Hence we get the following analogue of the Lemma 5 in \cite{renewalKendall}:
\begin{thm}\label{thm:5}
Suppose that $0\leq \theta <\alpha $. Let $\alpha W_Y(x) =   x^{\alpha} H_Y(x) +  x^{\alpha} \overline{F}_Y(x)$ and $\overline{F}_Y(x) \in RV_{\theta-\alpha}$.
Then, for $x \rightarrow \infty$, we have:

(i) $\overline{F}_Y(x)/H_Y(x)\rightarrow \theta /(\alpha -\theta )$;

(ii) $\alpha W_Y(x)/ x^{\alpha}H_Y(x)\rightarrow \alpha/(\alpha -\theta )$;

(iii) $\overline{G}_Y(x)/H_Y(x)\rightarrow \alpha /(\alpha -\theta )$ ;

(iv) $\overline{G}_Y(x)/\overline{F}_Y(x)\rightarrow \alpha / \theta $ .
\end{thm}

\noindent{\bf Proof.}
To see the results (i)-(iv) it is sufficient to notice that the following conditions are equivalent:
    $$
    x^{\alpha} H_Y(x) \in RV_{\theta } \iff 
    \overline{F}_Y(x) \in RV_{\theta-\alpha} \iff 
    \overline{G}_Y(x) \in RV_{\theta-\alpha} \iff
    W_Y(x) \in RV_{\theta }.
    $$ 
and apply mentioned Lemma 5 (\cite{renewalKendall}).
\qed
\begin{cor}
If $x^{\alpha} H_{Y}(x)\in RV_{0}$, then we have
$$
\frac{\overline{F}_{Y}(x)}{H_{Y}(x)}\rightarrow 0\text{, }\frac{
\overline{G}_{Y}(x)}{H_{Y}(x)}\rightarrow 1\text{, }\frac{W_{Y}(x)
}{x^{\alpha} H_{Y}(x)}\rightarrow \frac{1}{\alpha }\text{.} 
$$
\end{cor}

Now we have the potential to prove the following theorem:

\begin{thm} Let  $\overline{G} \in RV_{\theta-\alpha}$, $0\leq \theta <\alpha$, and $(v_n)_n$ be a positive sequence such that $\lim\limits_{n\to\infty} n\left(1-G(v_n)\right)= 1$. Then
\begin{eqnarray*}
\lim\limits_{n \to \infty} 2^{n}\Phi _{n}^{0}(v_{n}t) = \left( 1+ \frac{\alpha - \theta}{\alpha} t ^{\theta - \alpha}\right) \exp\{ - t^{\theta-\alpha}\}
\end{eqnarray*}
\end{thm}
\noindent{\bf Proof.}
First notice that $\overline{G}(x)=\overline{G}_{Y}(x)$. By Theorem \ref{thm:wynikazero} we have
\begin{eqnarray*}
\Phi^0_{n}(t) = \mathbb{P}(X_{\tau_{0}^{+}} \leq t, \tau_{0}^{+} = n) = \frac{1}{2^{n}}G(t)^{n - 1}\Big[2n\left(F(t) - \frac{1}{2}\right) - (n - 1)G(t)\Big].
\end{eqnarray*}
In our notation we have
$$
\Phi _{n}^{0}(t)=\frac{G_{Y}^{n-1}(t)}{2^{n}}(n(F_{Y}(t)-G_{Y}(t))+G_{Y}(t))
$$
or equivalently by Proposition \ref{prop:2} we arrive at
\begin{eqnarray*}
\Phi _{n}^{0}(t) &=&\frac{G_{Y}^{n}(t)}{2^{n}}\left(1+n\frac{H_{Y}(t)}{G_{Y}(t)
}\right) \\
&=&\frac{G_{Y}^{n}(t)}{2^{n}}\left(1+n\frac{H_{Y}(t)}{\overline{G}_{Y}(t)}
\frac{\overline{G}_{Y}(t)}{G_{Y}(t)}\right),
\end{eqnarray*}
since
$H_{Y}(t)=H(t)$.

Assuming that $t^{\alpha} H_{Y}(t) \in RV_{\theta },0<\theta <\alpha $, we have $\overline{G}\in RV_{\theta -\alpha }$
and
\[
\Phi _{n}^{0}(t)=\frac{G_{Y}^{n}(t)}{2^{n}}\left(1+\frac{\alpha -\theta }{\alpha }
(1+o(1))n\frac{\overline{G}_{Y}(t)}{G_{Y}(t)}\right)
\]
We choose $(v_{n})$ so that $
\lim\limits_{n\to\infty} n(1-G_{Y}(v_{n}))\rightarrow 1$.
It yields
\[
\lim\limits_{n\to\infty} \frac{1-G_{Y}(v_{n}x)}{1-G_{Y}(v_{n})}\rightarrow x^{\theta -\alpha },
\]
which follows
\[
\lim\limits_{n\to\infty} n(1-G_{Y}(v_{n}x)) = x^{\theta -\alpha }
\]
and then 
\[
\lim\limits_{n\to\infty} G_{Y}^{n}(v_{n}x) =  \exp \left\{-x^{\theta -\alpha }\right\}.
\]

We use the formula above and replace $t$ by $v_{n}t$. We find, as $n\to \infty$, 
\begin{eqnarray*}
2^{n}\Phi_{n}^{0}(v_{n}t) &=& G_{Y}^{n}(v_{n}t)\left(1+\frac{\alpha -\theta }{\alpha }
(1+o(1))n\frac{\overline{G}_{Y}(v_{n}t)}{G_{Y}(v_{n}t)}\right) \\
&\rightarrow & \left(1+\frac{\alpha -\theta }{\alpha }t^{\theta -\alpha
}\right) \exp \left\{-t^{-(\alpha-\theta)}\right\}.
\end{eqnarray*}
\qed

In the same manner we investigate asymptotic behaviour of the maxima distribution by regular variation:
\begin{thm}
Let  $\overline{G} \in RV_{\theta-\alpha}$, $0\leq \theta <\alpha$, and $(v_n)_n$ be a positive sequence such that $\lim\limits_{n\to\infty} n\left(1-G(v_n)\right)= 1$. Then
\begin{eqnarray*}
\lim\limits_{n\to\infty} \Phi (n,v_{n}t) = \left(1+\left(1-\frac{\theta }{\alpha }\right)t^{ -(\alpha-\theta) }\right) \exp \left\{-t^{-(\alpha-\theta) }\right\}.
\end{eqnarray*}
\end{thm}
\noindent{\bf Proof.}
By Lemma \ref{lem:14} we have
\begin{eqnarray*}
\Phi (n,t) &=&P\left(\max_{0\leq i\leq n}X_{i}\leq t\right) \\
&=&A(t)\frac{1}{2^{n}}+nB(t)\overline{G}(t)G^{n}(t)+(B(t)+C(t))G^{n}(t),
\end{eqnarray*}
where 
\begin{eqnarray*}
A(t) &=&1+\frac{H(t)}{(2G(t)-1)^{2}}-\frac{G(t)}{2G(t)-1} \\
B(t) &=&\frac{H(t)}{(2G(t)-1)\overline{G}(t)} \\
C(t) &=&\frac{G(t)}{2G(t)-1}-\frac{H(t)}{\overline{G}(t)}\frac{%
G(t)}{(2G(t)-1)^{2}}
\end{eqnarray*}

Clearly we have
\begin{eqnarray*}
B(t)+C(t) &=&\frac{H(t)}{(2G(t)-1)\overline{G}(t)}\left(1-\frac{G(t)}{
(2G(t)-1)}\right)+\frac{G(t)}{2G(t)-1} \\
&=&-\frac{H(t)}{(2G(t)-1)^{2}\overline{G}(t)}\overline{G}(t)+%
\frac{G(t)}{2G(t)-1},
\end{eqnarray*}

and $2G(t)-1\rightarrow 1$ as $t\rightarrow \infty $, 

\bigskip 

We have
\[
B(t)\rightarrow \frac{\alpha -\theta }{\alpha }=1-\frac{\theta }{\alpha }
\]
and
\[
A(t)=\frac{H(t)}{(2G(t)-1)^{2}}-\frac{\overline{G}(t)}{(2G(t)-1)}.
\]
For $t\to\infty$, it follows that
\[
\frac{A(t)}{\overline{G}(t)}=\frac{H(t)}{\overline{G}
(t)(2G(t)-1)^{2}}-\frac{1}{(2G(t)-1)}\rightarrow \frac{\alpha -\theta }{
\alpha } - 1 = - \frac{\theta }{\alpha }.
\]
Now consider $C(t)$. We have, for $t\to\infty$,
\begin{eqnarray*}
C(t) &=&\frac{G(t)}{2G(t)-1}-\frac{H(t)}{\overline{G}(t)}\frac{%
G(t)}{(2G(t)-1)^{2}} \\
&\rightarrow &1-\frac{\alpha -\theta }{\alpha }=\frac{\theta }{\alpha }.
\end{eqnarray*}

\bigskip 

As before we assume that $H\in RV_{\theta }$ so that $\overline{G}(t)\in
RV_{\theta-\alpha}$. We choose $(v_{n})$ so that 
$n(1-G(v_{n}))\rightarrow 1$.

We have 
\[
\lim\limits_{n\to\infty} \frac{1-G(v_{n} t)}{1-G(v_{n})} = t^{-(\alpha-\theta)}
\]%
so that%
\[
\lim\limits_{n\to\infty} n(1-G(v_{n} t)) = t^{-(\alpha-\theta)}
\]%
and then 
\[
\lim\limits_{n\to\infty} G^{n}(v_{n}t) = \exp\left\{ -t^{-(\alpha-\theta)}\right\}.
\]

Using $A(t)/\overline{G}(t) = c = - \theta/ \alpha$, we arrive at $nA(v_{n}t)\sim cn
\overline{G}(v_{n}t) = ct^{-(\alpha-\theta) }$.
If we take the formula 
\[
\Phi (n,t)=A(t)\frac{1}{2^{n}}+nB(t)\overline{G}%
(t)G^{n}(t)+(B(t)+C(t))G^{n}(t)
\]
and replace $t$ by $v_{n}t$ we find that as $n\rightarrow \infty $, 
\[
\Phi (n,v_{n}t) = ct^{-(\alpha-\theta) }(1+o(1))\frac{1}{n2^{n}}+\left( 1+ \left(1-\frac{\theta }{\alpha }
\right) t^{-(\alpha-\theta)}\right) \exp\left\{ -t^{-(\alpha-\theta)}\right\} (1+o(1))
\]
and finally, since $\lim\limits_{n \to \infty} A(v_{n}t)\frac{1}{2^{n}} = 0$, we have
$$
\lim\limits_{n\to\infty} \Phi (n,v_{n}t)  = \left( 1+ \left(1-\frac{\theta }{\alpha }
\right) t^{-(\alpha-\theta)}\right) \exp\left\{ -t^{-(\alpha-\theta)}\right\},
$$
which ends the proof.
\qed
\bigskip

{\bf Acknowledgements.} This paper is a part of project "First order Kendall maximal autoregressive processes and their applications", Grant no POIR.04.04.00-00-1D5E/16, which is carried out within the POWROTY/REINTEGRATION programme of the Foundation for Polish Science co-financed by the European Union under the European Regional Development Fund.

\end{document}